\documentclass[10pt,leqno]{amsart}
\usepackage{amsmath, amsthm, xcolor, graphicx, amssymb, mathtools, chngcntr, graphicx,  hyperref}

\newcounter{example}[section]
\newenvironment{example}[1][]{\refstepcounter{example}\par\medskip
	\noindent \textbf{Example~\theexample. #1} \rmfamily}{\medskip}

\newcommand{\be}{\begin{equation}}
	\newcommand{\ee}{\end{equation}}
\newcommand{\ben}{\begin{eqnarray*}}
	\newcommand{\een}{\end{eqnarray*}}

\newtheorem{theorem}{Theorem}[section]
\newtheorem{lemma}[theorem]{Lemma}
\newtheorem{corollary}[theorem]{Corollary}
\newtheorem{proposition}[theorem]{Proposition}
\newtheorem{remark}[theorem]{Remark}
\newtheorem{conjecture}[theorem]{Conjecture}
\newtheorem{definition}[theorem]{Definition}

\newcommand{\bt}{\begin{theorem}}
	\newcommand{\et}{\end{theorem}}
\newcommand{\bl}{\begin{lemma}}
	\newcommand{\el}{\end{lemma}}
\newcommand{\bcl}{\begin{corollary}}
	\newcommand{\ecl}{\end{corollary}}
\newcommand{\bex}{\begin{example}}
	\newcommand{\eex}{\end{example}}
\newcommand{\brem}{\begin{remark}}
	\newcommand{\erem}{\end{remark}}
\newcommand{\bed}{\begin{definition}}
	\newcommand{\eed}{\end{definition}}

\numberwithin{equation}{section}
\numberwithin{theorem}{section}

\renewcommand{\theexample}{\thesection.\arabic{example}}

\begin{document}
\title{On semimonotone matrices of exact order two}
	
	\author[B. P. Chauhan]{Bharat Pratap Chauhan}
	\address{(B. P. Chauhan) Department of Mathematics, Indian Institute of Technology Gandhinagar, Gujarat 382355, India}
	\email{bharat024pratap@gmail.com; bc102@snu.edu.in}
	
	\author[D. Dubey]{Dipti Dubey}
	\address{(D. Dubey) Department of Mathematics, Shiv Nadar Institution of Eminence, Delhi NCR 201314, India}
	\email{dipti.dubey@snu.edu.in}
	
	\maketitle
	
	\let\thefootnote\relax
	
	\begin{abstract}
In this paper, we introduce the notion of (strictly) semimonotone matrices of exact order $k$, where \( 0\leq k\leq n \), and explore their properties. We fully characterize the $3 \times 3$ (strictly) semimonotone matrices of exact order \( 2 \), and show that the class of $3 \times 3$  semimonotone matrices of exact order \( 2 \) forms a subclass of inverse $\mathbf{Z}$-matrices. We further investigate $ n\times n$ (strictly) semimonotone matrices of exact order $2$, with emphasis on their identification and construction, and establish that every $n\times n$ semimonotone $\mathbf{Z}$-matrix of exact order $2$ is invertible. Additionally, we show that when $n-k=1$, the class of (strictly) semimonotone matrices of exact order $k$ is a subclass of $\mathbf{Z}$-matrices.
	\end{abstract} 

	\noindent 
	
	\section{Introduction}
	
	Semimonotone matrices, defined as real square matrices for which the operation $Ax$ does not negate all positive entries of any nonzero, entrywise nonnegative vector $x$, are well studied (see \cite{cottle2009linear, Eav, Kar,tsatsomeros2019semimonotone}). Strictly semimonotone matrices ensure that $Ax$ neither negates nor nullifies all positive entries of such vectors. In addition to their pivotal role in linear complementarity theory \cite{cottle2009linear}, (strictly) semimonotone matrices are significant as they contain several well-known matrix classes, including (positive) nonnegative, $\mathbf{P_0}$ ($\mathbf{P}$), and (strictly) copositive matrices.
	
	The challenges associated with constructing and detecting semimonotone and strictly semimonotone matrices lead to exploring almost (strictly) semimonotone matrices. Tsatsomeros and Wendler introduced the class of almost (strictly) semimonotone matrices in \cite{tsatsomeros2019semimonotone, wendler2019almost}. An almost (strictly) semimonotone matrix $A$ is a matrix which is not (strictly) semimonotone but whose proper principal submatrices are (strictly) semimonotone. Apart from playing a crucial role in finding a new characterization of (strictly) semimonotone matrices, almost semimonotone matrices possess several interesting matrix-theoretic properties (see \cite{tsatsomeros2019semimonotone,wendler2019almost, chauhan2023almost,chauhan2024almost}). Notably, the inverse of an almost semimonotone matrix is a nonpositive matrix which is a $\mathbf{Z}$-matrix (a real square matrix is called a $\mathbf{Z}$-matrix if its off-diagonal entries are all nonpositive). Nonsingular matrices whose inverses are $\mathbf{Z}$-matrices are called inverse $\mathbf{Z}$-matrices, and the problem of determining which nonsingular matrices have inverses that are $\mathbf{Z}$-matrices is called the inverse $\mathbf{Z}$-matrix problem \cite{nabben}. So, almost semimonotone matrices form a subclass of inverse $\mathbf{Z}$-matrices. Since the inverse of a \( \mathbf{Z} \cap \mathbf{P} \)-matrix is always nonnegative (see \cite{Fid}), it follows that the inverse of a general \( \mathbf{Z} \)-matrix need not be a \( \mathbf{Z} \)-matrix. It is therefore interesting to determine for which classes of \( \mathbf{Z} \)-matrices the inverse retains the \( \mathbf{Z} \)-matrix property. 
	
	Similarly to the notion of a (strictly) copositive matrix of exact order $k$, $0 \leq k \leq n$ (see \cite{Val}), with the motivation to find a new characterization of (strictly) semimonotone matrices, we introduce the notion of (strictly) semimonotone matrices of exact order $k$, $0 \leq k \leq n$, which generalizes semimonotone and almost semimonotone matrices. We characterize $3\times 3$ semimonotone matrices of exact order $2$ and show that they are precisely the nonsingular $\mathbf{Z}$-matrices whose inverse is also a $\mathbf{Z}$-matrix. We extend this study to $n\times n$ semimonotone matrices of exact order $2$ and prove certain interesting properties. Finally, we explore semimonotone matrices of general exact order $k$ and show that, whenever $n-k=1$, they form a subclass of $\mathbf{Z}$-matrices.
	
	This paper is organized as follows. Section~\ref{preliminaries} introduces relevant definitions and notations. Section~\ref{mainsectioneo2} presents the semimonotone matrices of exact order $2$ and discusses the matrix-theoretic properties of these matrices. Section~\ref{generalresult} explores semimonotone matrices of exact order $k$, $2<k<n$.
	
	\section{Preliminaries}\label{preliminaries}
	We consider matrices and vectors with real entries. For any matrix \( A \in \mathbb{R}^{m \times n} \), \( a_{ij} \) denotes its \( i \)th row and \( j \)th column entry. We use \( A^{T} \) to denote the transpose of \( A \), \( \mathrm{tr}(A) \) for the trace of \( A \), \( \mathrm{adj}(A) \) for the adjugate matrix of \( A \), and \( \det A \) for the determinant of \( A \).
	
	For any index set \( \alpha \subseteq \{1,\ldots,n\} \), let \( \bar{\alpha} \) denote its complement in \( \{1,\ldots,n\} \). If \( A \) is a matrix of order \( n \), and \( \alpha, \beta \subseteq \{1,\ldots,n\} \), then \( A_{\alpha \beta} \) denotes the submatrix of \( A \) formed by selecting rows indexed by \( \alpha \) and columns indexed by \( \beta \). If \( \alpha = \beta \), then \( A_{\alpha \alpha} \) is called a \textit{principal submatrix} of \( A \), and \( \det A_{\alpha \alpha} \) is referred to as a \textit{principal minor} of \( A \).
	
	Let \( \alpha \subseteq \{1, \ldots, n\} \), and suppose \( A \in \mathbb{R}^{n \times n} \) is partitioned as
	\[
	A =
	\begin{bmatrix}
		A_{\alpha\alpha} & A_{\alpha\bar{\alpha}} \\
		A_{\bar{\alpha}\alpha} & A_{\bar{\alpha}\bar{\alpha}}
	\end{bmatrix}.
	\]
	If \( A_{\alpha \alpha} \) is invertible, then the \textit{Schur complement} of \( A_{\alpha \alpha} \) in \( A \), relative to the index set \( \alpha \), is defined as
	\[
	A / A_{\alpha \alpha} = A_{\bar{\alpha} \bar{\alpha}} - A_{\bar{\alpha} \alpha} A_{\alpha \alpha}^{-1} A_{\alpha \bar{\alpha}}.
	\]
	
	A matrix \( A \in \mathbb{R}^{n \times n} \) can be categorized based on various properties:
	
	\smallskip
	
	\( \mathbf{P_0} \), \( \mathbf{P} \), and \( \mathbf{Z} \)-matrices: A matrix \( A \) is a \( \mathbf{P_0} \)-matrix (resp., \( \mathbf{P} \)-matrix) if all principal minors of \( A \) are nonnegative (resp., positive). A matrix is called a \( \mathbf{Z} \)-matrix if all its off-diagonal entries are nonpositive.
	
	\smallskip
	
	hidden $\mathbf{Z}$-matrix: A matrix $A\in \mathbb{R}^{n\times n}$ is said to be a hidden $\mathbf{Z}$-matrix if there exist $\mathbf{Z}$-matrices X, Y $\in \mathbb{R}^{n\times n}$ such that  $AX = Y$ and $r^{T} X + s^{T} Y > 0$ for some  $r, s \in \mathbb{R}^{n}_{+}$.
	
	\smallskip
	
	Copositive and strictly copositive matrices: A matrix \( A \in \mathbb{R}^{n \times n} \) is called \textit{copositive} if \( x^{T} A x \geq 0 \) for all \( x \geq 0 \), and \textit{strictly copositive} if \( x^{T} A x > 0 \) for all \( x \geq 0 \), \( x \neq 0 \).
	
	\smallskip
	
	Semimonotone, strictly semimonotone, and almost (strictly) semimonotone matrices: A matrix \( A \in \mathbb{R}^{n \times n} \) is \textit{semimonotone} (an \( \mathbf{E_0} \)-matrix) if for every \( 0 \neq x \geq 0 \), there exists \( i \) such that \( x_i > 0 \) and \( (Ax)_i \geq 0 \). It is \textit{strictly semimonotone} (an \( \mathbf{E} \)-matrix) if \( (Ax)_i > 0 \) under the same condition. A matrix is \textit{almost (strictly) semimonotone} if all proper principal submatrices of \( A \) are (strictly) semimonotone and there exists a vector \( x > 0 \) such that \( Ax \leq 0 \) (resp., \( Ax < 0 \)).
	
	\smallskip
	
	Given a matrix $A\in \mathbb{R}^{n\times n}$ and a vector $q\in \mathbb{R}^{n},$ we define the {\it feasible set} of the LCP $(q,A)$ to be the set  FEA$(q,A)$ $=\{z\in\mathbb{R}^{n} \;|\;z\geq 0,\;q+Az\geq 0\}$ and the {\it solution set} of LCP $(q,A)$ by SOL$(q,A)=\{z\in FEA(q,A)\;|\;z^{T}(q+Az)$ $=0\}.$ A matrix $A\in \mathbb{R}^{n\times n}$  is called a {\it $\mathbf{Q}$-matrix} if for every $q\in \mathbb{R}^{n},$ LCP $(q,A)$ has a solution. We say that a matrix $A\in \mathbb{R}^{n\times n}$  is  a {\it $\mathbf{Q_{0}}$-matrix} if for any $q\in \mathbb{R}^{n},$ LCP $(q,A)$ has a feasible solution implies that LCP $(q,A)$ has a solution.
	
	\smallskip
	
	The following known results concerning almost semimonotone matrices will be used later.
	
	\begin{theorem}[\cite{wendler2019almost}, Theorem 3.3]\label{preliminary:theo3.3}
		A matrix $A \in \mathbb{R}^{2 \times 2}$ is almost (strictly) semimonotone if and only if $A = \begin{bmatrix} a & -b \\ -c & d \end{bmatrix}$ where $b, c > 0$, $a, d \geq 0$ (resp., $a, d > 0$), and $\det A < 0$ (resp., $\det A \leq 0$).
	\end{theorem}
	
	\begin{theorem}[\cite{tsatsomeros2019semimonotone}, Corollary 5.4]\label{preliminary:cor5.4}
		If $A \in \mathbb{R}^{n \times n}$ is almost semimonotone, then $A^{-1}$ exists and $A^{-1} \leq 0.$
	\end{theorem}
	
	\begin{theorem}[\cite{chauhan2023almost}, Theorems 3.6 and 3.7]\label{preliminary:theo3.6&3.7}
		If $A \in \mathbb R^{n \times n}$ is an almost (strictly) semimonotone matrix, then each row and each column of matrix $A$ has a negative entry.
	\end{theorem}
	
	\section{Semimonotone Matrices of Exact Order 2}\label{mainsectioneo2}
	
	In this section, we focus primarily on (strictly) semimonotone matrices of exact order \(k=2 \), as the analysis becomes significantly more challenging for larger values of \( k \). Constructing explicit examples of such matrices for \( k > 2 \) is nontrivial, especially in higher-dimensional cases. We begin this section by introducing the concept of (strictly) semimonotone matrices of exact order \( k \).
	
	\begin{definition}\label{definition:exact_order_k_matrix}
		A matrix \( A \in \mathbb{R}^{n \times n} \) is said to be a \textit{(strictly) semimonotone matrix of exact order \( k \)}, where \( 0 \leq k \leq n \), if every principal submatrix of \( A \) of order \( n - k \) is (strictly) semimonotone, but no principal submatrix of order \( r \), where \( n - k < r \leq n \), is (strictly) semimonotone.
	\end{definition}
	
	The class of semimonotone matrices of exact order \( k \) is denoted by \( \mathbf{E_{0_k}} \), and its elements are called \( \mathbf{E_{0_k}} \)-matrices. Similarly, the class of strictly semimonotone matrices of exact order \( k \) is denoted by \( \mathbf{E_k} \), and its elements are called \( \mathbf{E_k} \)-matrices.
	
	\begin{example}
		The following matrices are examples of an \( \mathbf{E_{0_2}} \)-matrix and an \( \mathbf{E_2} \)-matrix, respectively:
		\[
		A = \begin{bmatrix}
			0 & -1 & -1 \\
			-2 & 0 & -1 \\
			-3 & -4 & 0
		\end{bmatrix}
		\quad \text{and} \quad
		B = \begin{bmatrix}
			1 & -2 & -2 \\
			-2 & 1 & -2 \\
			-2 & -2 & 1
		\end{bmatrix}.
		\]
	\end{example}
	
	\begin{remark}
		From Definition~\ref{definition:exact_order_k_matrix}, it is clear that every (strictly) semimonotone matrix is a (strictly) semimonotone matrix of exact order \( 0 \), and that every almost (strictly) semimonotone matrix is a (strictly) semimonotone matrix of exact order \( 1 \).
	\end{remark}
	
	The classes of (strictly) semimonotone matrices of exact orders \( 0 \) and \( 1 \) are well studied in the literature. In the upcoming subsections, we present characterization theorems for $3\times 3$ (strictly) semimonotone matrices of exact order \( 2 \), and then discuss a few results concerning matrices of general exact order \( 2 \). 
	
	\subsection{Semimonotone Matrices of Exact Order Two in \( \mathbb{R}^{3 \times 3} \)}
	
	In this subsection, we investigate the structure of \( \mathbf{E_{0_2}} \) (resp., \( \mathbf{E_2} \))-matrices within the class of \( 3 \times 3 \) real matrices and present a complete characterization in terms of their sign patterns and principal minors.
	
	\begin{theorem}\label{theorem:3x3structure}
		Let \( A \in \mathbb{R}^{3 \times 3} \). If \( A \) is an \( \mathbf{E_{0_2}} \) (resp., \( \mathbf{E_2} \))-matrix, then the following hold:
		\begin{itemize}
			\item[(a)] \( A \) is a \( \mathbf{Z} \)-matrix and has the form
			\[
			A =
			\begin{bmatrix}
				a & -b & -c \\
				-d & e & -f \\
				-g & -h & i
			\end{bmatrix},
			\]
			where \( a, e, i \geq 0 \) (resp., \( a, e, i > 0 \)) and \( b, c, d, f, g, h > 0 \).
			
			\item[(b)] All principal minors of \( A \) of order \( 2 \) are negative (resp., nonpositive).
		\end{itemize}
	\end{theorem}
	
	\begin{proof}
		$(a)$. Suppose \( A \in \mathbb{R}^{3 \times 3} \) is an \( \mathbf{E_{0_2}} \) (resp., \( \mathbf{E_2} \))-matrix. Then every principal submatrix of \( A \) of order \( 1 \) is (strictly) semimonotone, while no principal submatrix of order \( 2 \) or \( 3 \) is (strictly) semimonotone. It implies that all diagonal entries of \( A \) are nonnegative (resp., positive), and that every principal submatrix of order \( 2 \) is almost (strictly) semimonotone.
		
		By Theorem~\ref{preliminary:theo3.3}, each principal submatrix of \( A \) of order $2$ has negative off-diagonal entries, which further implies that all off-diagonal entries of \( A \) are negative. That is , \( A \) is a \( \mathbf{Z} \)-matrix and $A$ can be written in the following form:
		$$
		A =
		\begin{bmatrix}
			a & -b & -c \\
			-d & e & -f \\
			-g & -h &i		
		\end{bmatrix},
		$$
		where \( a, e, i \geq 0 \) (resp., \( a, e, i > 0 \)) and \( b, c, d, f, g, h > 0 \).
		
		$(b)$. From the definition of $A$, all the principal submatrices of \( A \) of order $2$ are almost (strictly) semimonotone. Then, by Theorem~\ref{preliminary:theo3.3}, the determination of all its principal submatrices of $A$ of order $2$ is negative (resp., nonpositive).
	\end{proof}
	
	\bcl
	Let \( A \in \mathbb{R}^{3 \times 3} \) be an \( \mathbf{E_{0_2}} \) (resp., \( \mathbf{E_2} \))-matrix. Then \( A \) cannot be a triangular matrix, and hence \( A \) is irreducible.
	\ecl
	
	\begin{proof}
		By Theorem~\ref{theorem:3x3structure}$(a)$, every $\mathbf{E_{0_2}}$ (resp., $\mathbf{E_2}$)-matrix has all off-diagonal entries negative. The desired conclusion for $A$ follows immediately.
	\end{proof}
	In the next result, we present the relationship between
	\( \mathbf{E_{0_2}} \) (resp., \( \mathbf{E_2} \)) and
	\( \mathbf{Q_0} \)-matrices for matrices of order \(3\).
	
	\begin{corollary}\label{cor:E02-Q0}
		Let \( A \in \mathbb{R}^{3 \times 3} \).
		If \( A \) is an \( \mathbf{E_{0_2}} \) (resp., \( \mathbf{E_2} \))-matrix,
		then \( A \) is a \( \mathbf{Q_0} \)-matrix.
	\end{corollary}
	
	\begin{proof}
		By Theorem~\ref{theorem:3x3structure}\emph{(a)}, an
		\( \mathbf{E_{0_2}} \) (resp., \( \mathbf{E_2} \))-matrix of order \(3\)
		is a \( \mathbf{Z} \)-matrix.
		Moreover, every \( \mathbf{Z} \)-matrix is a hidden \( \mathbf{Z} \)-matrix.
		Hence, by \cite[Theorem~3.9]{dub}, \( A \) is a \( \mathbf{Q_0} \)-matrix.
	\end{proof}
	
	In the following result, we establish the existence of the inverse of an \( \mathbf{E_{0_2}} \)-matrix in \( \mathbb{R}^{3 \times 3} \), using the Schur complement.
	
	\begin{theorem}\label{theorem:3x3inverse}
		Let \( A \in \mathbb{R}^{3 \times 3} \) be an \( \mathbf{E_{0_2}} \)-matrix. Then \( \det A < 0 \), and hence \( A^{-1} \) exists and is a \( \mathbf{Z} \)-matrix. Moreover, \( A \) has exactly one negative eigenvalue.
	\end{theorem}
	
	\begin{proof} 
		Let $\alpha = \{1,2\}$, and let $A \in \mathbb{R}^{3 \times 3}$ be an $\mathbf{E_{0_2}}$-matrix.  
		By Theorem~\ref{theorem:3x3structure}(a), $A$ admits the representation
		\[
		A =
		\begin{bmatrix}
			A_{\alpha\alpha} & A_{\alpha\bar{\alpha}} \\
			A_{\bar{\alpha}\alpha} & A_{\bar{\alpha}\bar{\alpha}}
		\end{bmatrix}
		=
		\begin{bmatrix}
			a & -b & -c \\
			-d & e & -f \\
			-g & -h & i
		\end{bmatrix},
		\]
		where
		\[
		A_{\alpha\alpha} =
		\begin{bmatrix}
			a & -b \\
			-d & e
		\end{bmatrix},
		\qquad
		a,e,i \ge 0,
		\quad
		b,c,d,f,g,h > 0.
		\]
		
		By Theorem~\ref{theorem:3x3structure}(b), every principal minor of $A$ of order $2$ is negative. 
		In particular, $\det A_{\alpha\alpha} = ae - bd < 0,$ and hence $A_{\alpha\alpha}$ is nonsingular. Therefore, the Schur complement of $A_{\alpha\alpha}$ in $A$ is well defined. A direct computation gives
		\begin{align*}
			A / A_{\alpha\alpha}
			&= A_{\bar{\alpha}\bar{\alpha}}
			- A_{\bar{\alpha}\alpha} A_{\alpha\alpha}^{-1} A_{\alpha\bar{\alpha}} \\
			&= i
			- \begin{bmatrix} -g & -h \end{bmatrix}
			\begin{bmatrix} a & -b \\ -d & e \end{bmatrix}^{-1}
			\begin{bmatrix} -c \\ -f \end{bmatrix} \\
			&= i
			- \frac{1}{ae - bd}
			\begin{bmatrix} -g & -h \end{bmatrix}
			\begin{bmatrix} e & b \\ d & a \end{bmatrix}
			\begin{bmatrix} -c \\ -f \end{bmatrix} \\
			&= i - \frac{g(ce + bf) + h(cd + af)}{ae - bd}.
		\end{align*}
		
		Since $\det A_{\alpha\alpha} = ae - bd < 0$ and $a,e \ge 0$ together with
		$b,c,d,f,g,h > 0$ imply that
		\[
		g(ce + bf) + h(cd + af) > 0,
		\]
		it follows that
		\[
		A / A_{\alpha\alpha} > 0,
		\]
		and hence $\det(A / A_{\alpha\alpha}) > 0$. Therefore, by Schur’s determinantal
		formula (see \cite{crabtree1969identity}),
		\[
		\det A = \det A_{\alpha\alpha} \cdot \det(A / A_{\alpha\alpha}).
		\]
		Since $\det A_{\alpha\alpha} < 0$, we conclude that $\det A < 0$. In particular,
		$A$ is nonsingular, and hence $A^{-1}$ exists.
		
		Next, we compute the adjugate matrix of $A$ to determine $A^{-1}$. Given the structure of $A$, we have
		\[
		\text{adj}(A) = \begin{bmatrix}
			ei - fh & bi + ch & bf + ce \\
			di + fg & ai - cg & af + cd \\
			dh + eg & ah + bg & ae - bd
		\end{bmatrix}.
		\]
		
		All off-diagonal entries of $\text{adj}(A)$ are positive since $a, e, i \geq 0$ and $b, c, d, f, g, h > 0$. Using the identity 
		\[
		A^{-1} = \frac{1}{\det A} \text{adj}(A),
		\]
		and noting that $\det A < 0$, we conclude that all off-diagonal entries of $A^{-1}$ are negative. Therefore, $A^{-1}$ is a $\mathbf{Z}$-matrix.
		
		Next, we consider $\text{det}A$, which is negative, implying that $A$ has an odd number of negative eigenvalues. Since all complex eigenvalues of a $3\times 3$ matrix $A$ must appear in conjugate pairs and $\text{tr}(A) \geq 0$, it follows that $A$ has exactly one negative eigenvalue.
	\end{proof}

	
	\subsection{Semimonotone Matrices of Exact Order Two in \( \mathbb{R}^{n \times n} \)}
	
	In the preceding subsection, we examined the structure of \( \mathbf{E_{0_2}} \) (resp., \( \mathbf{E_2} \))-matrices within the class of \( 3 \times 3 \) real matrices. Here, we extend our investigation to \( n\times n \) real matrices of these classes. We begin by formally defining \( \mathbf{E_{0_2}} \) (resp., \( \mathbf{E_2} \))-matrices, followed by an exploration of their fundamental properties.
	
	\begin{definition}
		A matrix \( A \in \mathbb{R}^{n \times n} \) is said to be an \( \mathbf{E_{0_2}} \) (resp., \( \mathbf{E_2} \))-matrix if every principal submatrix of \( A \) of order \( n - 2 \) is (strictly) semimonotone, but no principal submatrix of order \( r \), where \( r = n - 1 \) or \( r = n \), is (strictly) semimonotone.
	\end{definition}
	
	\begin{example}\label{example:4.2}
		Consider the matrix
		\[
		A =
		\begin{bmatrix}
			1 & \frac{1}{2} & -1 & -1 \\
			\frac{1}{2} & 1 & -1 & -1 \\
			-1 & -1 & 1 & 0 \\
			-1 & -1 & 0 & 1
		\end{bmatrix}.
		\]
		Here, \( A \) is a \( 4 \times 4 \) real matrix. It is straightforward to verify that all principal submatrices of \( A \) of order \( 2 \) are semimonotone, while no principal submatrix of order \( 3 \) or \( 4 \) is semimonotone. Therefore, \( A \in \mathbf{E_{0_2}} \). 
		
		However, since not all off-diagonal entries of \( A \) are nonpositive, \( A \) is not a \( \mathbf{Z} \)-matrix. Interestingly, the inverse of \( A \) is a \( \mathbf{Z} \)-matrix:
		\[
		A^{-1} =
		\begin{bmatrix}
			\frac{4}{5} & -\frac{6}{5} & -\frac{2}{5} & -\frac{2}{5} \\
			-\frac{6}{5} & \frac{4}{5} & -\frac{2}{5} & -\frac{2}{5} \\
			-\frac{2}{5} & -\frac{2}{5} & \frac{1}{5} & -\frac{4}{5} \\
			-\frac{2}{5} & -\frac{2}{5} & -\frac{4}{5} & \frac{1}{5}
		\end{bmatrix}.
		\]
	\end{example}
	The above example shows that Theorem~\ref{theorem:3x3structure} does not extend to arbitrary \( n\times n \) real matrices. Nevertheless, it motivates the formulation of the following result, which addresses the general case.
	
	\begin{proposition}\label{proposition:4.10}
		Let  \( A \in \mathbb{R}^{n \times n} \), where \( n \geq 3 \) be a positive integer. If \( A \) is an \( \mathbf{E_{0_2}} \) (resp., \( \mathbf{E_2} \))-matrix, then the following hold:
		\begin{itemize}
			\item[(a)] All diagonal entries of \( A \) are nonnegative (resp., positive).
			
			\item[(b)] Each row and each column of \( A \) contains at least two negative entries.
		\end{itemize}
	\end{proposition}
	
	\begin{proof}
		Suppose \( A \in \mathbb{R}^{n \times n} \) is an \( \mathbf{E_{0_2}} \) (resp., \( \mathbf{E_2} \))-matrix. Then every principal submatrix of \( A \) of order \( n-2 \) is (strictly) semimonotone, while no principal submatrix of order \( n-1 \) or \( n \) is (strictly) semimonotone. Note that \( n - 2 \geq 1 \), so all principal submatrices of \( A \) of order 1 are (strictly) semimonotone. This implies that all diagonal entries of \( A \) are nonnegative (resp., positive), which proves part~$(a)$.
		
		To prove part~(b), we first show that each row of an \( \mathbf{E_{0_2}} \) (resp., \( \mathbf{E_2} \))-matrix \( A \) contains at least two negative entries. Consider the index set \( \alpha = \{1, \ldots, n\} \setminus \{n\} \), and let \( A_{\alpha \alpha} \) be the corresponding principal submatrix of order \( n - 1 \). Since every principal submatrix of order \( n-1 \) is almost (strictly) semimonotone, it follows from Theorem~\ref{preliminary:theo3.6&3.7} that each row of \( A_{\alpha\alpha} \) contains at least one negative entry. It implies that the first \( n-1 \) rows of \( A \) contain at least one negative entry. Similarly, consider \( \alpha = \{1, \ldots, n\} \setminus \{1\} \). The corresponding submatrix \( A_{\alpha\alpha} \) again satisfies the same property, implying that the last \( n-1 \) rows of \( A \) each contain at least one negative entry. Since every row of \( A \) appears in at least two such submatrices, each row of \( A \) must contain at least one negative entry.
		
		Now, suppose that some row of \( A \), say the \( i \)th row, contains exactly one negative entry, located in the \( j \)th column. Consider the index set \( \alpha = \{1,\ldots,n\} \setminus \{j\} \). Then the corresponding principal submatrix \( A_{\alpha \alpha} \) has a nonnegative \( i \)th row of \(A\). It contradicts the fact that \( A_{\alpha \alpha} \) is an almost (strictly) semimonotone matrix, which, by Theorem~\ref{preliminary:theo3.6&3.7}, implies that each row of \( A_{\alpha \alpha} \) must contain at least one negative entry. Therefore, every row of \( A \) must contain at least two negative entries.
		
		The fact that each column of \( A \) contains at least two negative entries follows from the observation that \( A \) is an \( \mathbf{E_{0_2}} \) (resp., \( \mathbf{E_2} \))-matrix if and only if \( A^{T} \) is also an \( \mathbf{E_{0_2}} \) (resp., \( \mathbf{E_2} \))-matrix (see, Proposition \ref{proposition:transpose}).
	\end{proof}
	
	In the following result, we show that every \( n\times n \) real matrix in \( \mathbf{Z} \cap \mathbf{E_{0_2}} \) is invertible, and we establish this using the Schur complement.
	
	\begin{theorem}\label{theorem:4.11}
		Let \( A \in \mathbb{R}^{n \times n} \cap \mathbf{Z} \) be an \( \mathbf{E_{0_2}} \)-matrix. Then the following statements hold:
		\begin{itemize}
			\item[(a)] All proper principal minors of \( A \) of order at most \( n - 2 \) are nonnegative.
			
			\item[(b)] All proper principal minors of \( A \) of order \( n - 1 \) are negative.
			
			\item[(c)] \( \det A < 0 \).
			
			\item[(d)] The inverse \( A^{-1} \) exists, and all diagonal entries of \( A^{-1} \) are positive.
		\end{itemize}
	\end{theorem}
	
	\begin{proof}
		Suppose that \( A \in \mathbb{R}^{n \times n} \cap \mathbf{Z} \) is an \( \mathbf{E_{0_2}} \)-matrix. Then every principal submatrix of \( A \) of order \( n-2 \) is semimonotone, while no principal submatrix of order \( n-1 \) or \( n \) is semimonotone. By \cite[Theorem~7]{chu2003semimonotone}, every principal submatrix of \( A \) of order \( n-2 \) is a \( \mathbf{P_0} \)-matrix. Therefore, all proper principal minors of \( A \) of order \( \leq n-2 \) are nonnegative, proving part~$(a)$.
		
		To establish part~(b), let $A$ be a $\mathbf{Z}$-matrix that is an
		$\mathbf{E_{0_2}}$-matrix. By definition, every principal submatrix of $A$ of
		order $n-2$ is semimonotone, whereas no principal submatrix of $A$ of order
		$n-1$ is semimonotone. Since, under the $\mathbf{Z}$-matrix assumption, the class of semimonotone
		matrices coincides with the class of $\mathbf{P_0}$-matrices
		(see \cite[Theorem~7]{chu2003semimonotone}), it follows that every principal
		submatrix of $A$ of order $n-1$ fails to be a $\mathbf{P_0}$-matrix. Hence, all
		proper principal minors of $A$ of order $n-1$ are negative.
		
		We now prove part~(c). Let \( \alpha \subseteq \{1,\ldots,n\} \) with \( |\alpha| = n-1 \), and consider the partitioned matrix
		\[
		A =
		\begin{bmatrix}
			A_{\alpha\alpha} & A_{\alpha\bar{\alpha}} \\
			A_{\bar{\alpha}\alpha} & A_{\bar{\alpha}\bar{\alpha}}
		\end{bmatrix}.
		\]
		By the definition of an \( \mathbf{E_{0_2}} \)-matrix, the principal submatrix \( A_{\alpha\alpha} \) of \( A \) of order \( n-1 \) is almost semimonotone. Therefore, by  Theorem~\ref{preliminary:cor5.4}, the inverse $A_{\alpha\alpha}^{-1}$ exists
		and satisfies $A_{\alpha\alpha}^{-1} \le 0$. Moreover, by the structural form
		of $A$, we have $0 \neq A_{\bar{\alpha}\alpha} \le 0$ and
		$0 \neq A_{\alpha\bar{\alpha}} \le 0$. Since the product of nonpositive
		matrices is again nonnegative in each entry, it follows that
		\begin{equation}\label{eq:common_3.8theorem_equation}
			A_{\bar{\alpha}\alpha} A_{\alpha\alpha}^{-1} A_{\alpha\bar{\alpha}} \le 0.
		\end{equation}
		
		Since \( A_{\alpha\alpha}^{-1} \) exists, the Schur complement of \( A_{\alpha\alpha} \) in \( A \) is well-defined and given by
		\[
		A / A_{\alpha\alpha} = A_{\bar{\alpha} \bar{\alpha}} - A_{\bar{\alpha} \alpha} A_{\alpha\alpha}^{-1} A_{\alpha \bar{\alpha}}.
		\]
		
		We now show that \( A / A_{\alpha\alpha} > 0 \) by considering the following two cases. By Proposition~\ref{proposition:4.10}$(a)$, all diagonal entries of \( A \) are nonnegative. Therefore, all diagonal entries are zero, or at least one is positive.
		
		\medskip
		
		\noindent
		\textbf{Case (i):} All diagonal entries of \( A \) are zero. Then \( A_{\bar{\alpha} \bar{\alpha}} = 0 \), and \( A_{\alpha\alpha} \) is an almost semimonotone $\mathbf{Z}$-matrix with zero diagonal. By \cite[Theorem~4.1]{chauhan2024almost}, \( A_{\alpha\alpha} \) is an irreducible monomial matrix; that is, each row and each column of \( A_{\alpha\alpha} \) contains exactly one negative entry. Moreover, by Proposition~\ref{proposition:4.10}$(b)$, each row and each column of \( A \) contains at least two negative entries. This implies that $A_{\bar{\alpha}\alpha} < 0$ and
		$A_{\alpha\bar{\alpha}} < 0$ entrywise. Moreover,
		$A_{\alpha\alpha}^{-1} \le 0$ is nonzero, and each row of
		$A_{\alpha\alpha}^{-1}$ contains at least one strictly negative entry, it follows that
		\[
		A_{\bar{\alpha}\alpha} A_{\alpha\alpha}^{-1} A_{\alpha\bar{\alpha}} < 0.
		\]
		Since $A_{\bar{\alpha}\bar{\alpha}} = 0$, the Schur complement satisfies
		\begin{align*}
			A / A_{\alpha\alpha}
			&= A_{\bar{\alpha}\bar{\alpha}}
			- A_{\bar{\alpha}\alpha} A_{\alpha\alpha}^{-1} A_{\alpha\bar{\alpha}} \\
			&= - A_{\bar{\alpha}\alpha} A_{\alpha\alpha}^{-1} A_{\alpha\bar{\alpha}} \\
			&> 0.
		\end{align*}
		
		\noindent
		\textbf{Case (ii):} At least one diagonal entry of \( A \) is positive. Without loss of generality, suppose \( A_{\bar{\alpha} \bar{\alpha}} > 0 \). As previously established (see \eqref{eq:common_3.8theorem_equation}),
		\[
		A_{\bar{\alpha} \alpha} A_{\alpha\alpha}^{-1} A_{\alpha \bar{\alpha}} \leq 0,
		\]
		and hence
		\[
		A / A_{\alpha\alpha} = A_{\bar{\alpha} \bar{\alpha}} - A_{\bar{\alpha} \alpha} A_{\alpha\alpha}^{-1} A_{\alpha \bar{\alpha}} > 0.
		\]
		
		Combining both cases, we conclude that \( A / A_{\alpha\alpha} > 0 \). In particular,
		\[
		\det(A / A_{\alpha\alpha}) > 0.
		\]
		
		Additionally, since \( \det A_{\alpha\alpha} < 0 \), it follows from Schur’s determinantal formula that
		\[
		\det A = \det A_{\alpha\alpha} \cdot \det(A / A_{\alpha\alpha}) < 0.
		\]
		This completes the proof of part~(c).
		
		To establish part~(d), consider the diagonal entries of the adjugate matrix \( \mathrm{adj}(A) \), which are given by
		\[
		\left( \mathrm{adj}(A) \right)_{ii} = (-1)^{i+i} \det A_{\beta\beta} = \det A_{\beta\beta},
		\]
		where \( \beta = \{1, \ldots, n\} \setminus \{i\} \) for \( i = 1, 2, \dots, n \). Since \( \det A < 0 \), the matrix \( A \) is invertible. For any real, invertible matrix \( A \), the identity $\mathrm{adj}(A) = (\det A) A^{-1}$ implies that the diagonal entries of \( A^{-1} \) are given by
		\[
		(A^{-1})_{ii} = \frac{\det A_{\beta\beta}}{\det A}, \quad \text{for } i = 1, 2, \dots, n.
		\]
		Since \( \det A < 0 \) and \( \det A_{\beta\beta} < 0 \) for each \( \beta \), it follows that \( (A^{-1})_{ii} > 0 \) for all \( i = 1, 2, \dots, n \). Therefore, all diagonal entries of \( A^{-1} \) are positive.
		
		Moreover, from part~(c), since \( \det(A / A_{\alpha\alpha}) > 0 \), the Schur complement \( A / A_{\alpha\alpha} \) is invertible and satisfies \( (A / A_{\alpha\alpha})^{-1} > 0 \). Therefore, by \cite[Proposition~2.3.5]{cottle2009linear}, the inverse of \( A \) can be expressed as
		\[
		A^{-1} = \begin{bmatrix}
			A_{\alpha\alpha}^{-1} + A_{\alpha\alpha}^{-1} A_{\alpha \bar{\alpha}} (A / A_{\alpha\alpha})^{-1} A_{\bar{\alpha} \alpha} A_{\alpha\alpha}^{-1} & -A_{\alpha\alpha}^{-1} A_{\alpha \bar{\alpha}} (A / A_{\alpha\alpha})^{-1} \\
			-(A / A_{\alpha\alpha})^{-1} A_{\bar{\alpha} \alpha} A_{\alpha\alpha}^{-1} & (A / A_{\alpha\alpha})^{-1}
		\end{bmatrix}.
		\]
		
		From this expression, we observe that
		\[
		-(A / A_{\alpha\alpha})^{-1} A_{\bar{\alpha} \alpha} A_{\alpha\alpha}^{-1} \leq 0 \quad \text{and} \quad -A_{\alpha\alpha}^{-1} A_{\alpha \bar{\alpha}} (A / A_{\alpha\alpha})^{-1} \leq 0.
		\]
		Additionally, the matrix \( (A / A_{\alpha\alpha})^{-1} \), which appears as a diagonal block in \( A^{-1} \), is strictly positive.
	\end{proof}
	
	\begin{remark}
		We observe from Theorem~\ref{theorem:4.11} that for any \(\mathbf{Z} \cap \mathbf{E_{0_2}} \)-matrix, the Schur complement of \( A_{\alpha\alpha} \) in \( A \) exists and is positive for every \( \alpha = \{1, \ldots, n\} \setminus \{i\} \), where \( i = 1, 2, \dots, n \).
	\end{remark}
	
	\begin{example}
		Consider the matrix
		\[
		A =
		\begin{bmatrix}
			1 & \frac{1}{2} & -1 & -1 \\
			\frac{1}{3} & 1 & -1 & -1 \\
			-1 & -1 & 1 &  0 \\
			-1 & -1 & 0 & 1
		\end{bmatrix} \quad \text{and} \quad
		A^{-1} =
		\begin{bmatrix}
			\frac{2}{3} & -1 & -\frac{1}{3} & -\frac{1}{3} \\
			-\frac{10}{9} & \frac{2}{3} & -\frac{4}{9} & -\frac{4}{9} \\
			-\frac{4}{9} & -\frac{1}{3} & \frac{2}{9} & -\frac{7}{9} \\
			-\frac{4}{9} & -\frac{1}{3} & -\frac{7}{9} & \frac{2}{9}
		\end{bmatrix}.
		\]
		Here, \( A \) belongs to the class \( \mathbf{E_{0_2}} \). Note that \( A \) is neither symmetric nor a \( \mathbf{Z} \)-matrix. However, the inverse \( A^{-1} \) is a \( \mathbf{Z} \)-matrix.
	\end{example}
	
	\vspace{0.5em}
	Motivated by numerous examples and previous observations, we propose the following conjectures:
	
	\begin{conjecture}\label{conjecture:1}
		Let \( A \in \mathbb{R}^{n \times n} \cap \mathbf{Z} \) be an \( \mathbf{E_{0_2}} \)-matrix. Then, the leading principal diagonal block of $A^{-1}:$
		\[
		A_{\alpha\alpha}^{-1} + A_{\alpha\alpha}^{-1} A_{\alpha \bar{\alpha}} (A / A_{\alpha\alpha})^{-1} A_{\bar{\alpha} \alpha} A_{\alpha\alpha}^{-1},
		\]
		as described in the proof of Theorem~\ref{theorem:4.11}, is a \( \mathbf{Z} \)-matrix. Moreover, \(A\) has exactly one negative eigenvalue.
	\end{conjecture}
	
	\begin{conjecture}
		Let \( A \in \mathbb{R}^{n \times n} \) be an \( \mathbf{E_{0_2}} \)-matrix. Then, \( \det(A) < 0 \), \( A^{-1} \) exists and is a \( \mathbf{Z} \)-matrix.
	\end{conjecture}
	
	\vspace{0.5em}
	If the above conjectures hold, then the class of semimonotone matrices of exact order two (\( \mathbf{E_{0_2}} \)) forms a subclass of inverse \( \mathbf{Z} \)-matrices.
	
	
	\section{Semimonotone Matrices of Exact Order k}\label{generalresult}
	
	This section examines semimonotone matrices of exact order \( k \) in \( \mathbb{R}^{n\times n} \), a class of matrices for which explicit constructions become increasingly complex as \( k \) grows. We present representative examples for \( k = 2 \) and \( k = 3 \), which also provide insight into the structure of \( \mathbf{E_{0_k}} \)-matrices. 
	These examples motivate the observations made in the subsequent remark.
	
	\begin{example}
		\begin{itemize}
			\item[(a)] Semimonotone matrices of exact order \( 3 \):
			\[
			A = \begin{bmatrix}
				1 & -2 & -3 & -2 \\
				-3 & 1 & -2 & -3 \\
				-2 & -3 & 1 & -4 \\
				-3 & -2 & -3 & 1
			\end{bmatrix}, \quad
			B = \begin{bmatrix}
				1 & -\tfrac{1}{2} & -1 & -1 & -1 \\
				0 & 1 & -1 & -1 & -1 \\
				-1 & -1 & 1 & -1 & -1 \\
				-1 & -1 & -1 & 1 & 0 \\
				-1 & -1 & -1 & 0 & 1
			\end{bmatrix}.
			\]
			
			\item[(b)] Semimonotone matrices of exact order \( 2 \):
			\[
			A = \begin{bmatrix}
				1 & \tfrac{1}{2} & -1 & -1 \\
				0 & 1 & -1 & -1 \\
				-1 & -1 & 1 & 0 \\
				-1 & -1 & 0 & 1
			\end{bmatrix}, \quad
			B = \begin{bmatrix}
				1 & 0 & 0 & -1 & -1 \\
				-1 & 2 & 0 & 0 & -1 \\
				-1 & -\tfrac{1}{2} & 1 & 0 & 0 \\
				0 & -1 & -1 & 1 & 0 \\
				0 & -\tfrac{1}{2} & -\tfrac{1}{2} & -1 & 1
			\end{bmatrix}.
			\]
		\end{itemize}
		
		Each matrix is semimonotone of the specified exact order and is invertible with negative determinant. 
		In all cases, the inverse is a \( \mathbf{Z} \)-matrix with strictly positive diagonal entries. 
		Moreover, each matrix has exactly one negative eigenvalue.
	\end{example}
	
	\begin{remark}
		Motivated by the above examples and observations, we pose the following
		questions.
		
		Let $n \ge 3$ and $1 \le k < n$, and let $A \in \mathbb{R}^{n \times n}$.
		
		\begin{itemize}
			\item Is it true that if $A \in \mathbf{E_{0_k}}$ (resp., $\mathbf{E_k}$), then
			each row and each column of $A$ contains at least $k$ negative entries?
			
			\item Is it true that if $A \in \mathbf{E_{0_k}}$, then $\det A < 0$ and $A^{-1}$
			exists and is a $\mathbf{Z}$-matrix?
		\end{itemize}
	\end{remark}
	
	\begin{remark}
		One may ask whether the class \( \mathbf{E_{0_k}} \) is closed under matrix addition or multiplication. The following example shows that this is not the case.
		
		\[
		A = \begin{bmatrix}
			10 & -1 & -1 \\
			-1 & 0 & -1 \\
			-1 & -1 & 0
		\end{bmatrix}, \quad 
		B = \begin{bmatrix}
			1 & -2 & -2 \\
			-2 & 1 & -2 \\
			-2 & -2 & 1
		\end{bmatrix}.
		\]
		
		Both \( A \) and \( B \) are \( \mathbf{E_{0_2}} \)-matrices. However, neither their sum nor their product belongs to \( \mathbf{E_{0_2}} \):
		\[
		A + B = 
		\begin{bmatrix}
			11 & -3 & -3 \\
			-3 & 1 & -3 \\
			-3 & -3 & 1
		\end{bmatrix}, \quad 
		AB = 
		\begin{bmatrix}
			14 & -19 & -19 \\
			1 & 4 & 1 \\
			1 & 1 & 4
		\end{bmatrix}.
		\]
	\end{remark}
	
	\begin{remark}
		It is known that the sum of a semimonotone matrix and a nonnegative matrix is again semimonotone \cite[Theorem 4.1]{tsatsomeros2019semimonotone}. However, a similar result does not hold for the class of \( \mathbf{E_{0_k}} \)-matrices, as demonstrated below.
		
		\[
		A = \begin{bmatrix}
			0 & -1 & -1 \\
			-1 & 0 & -1 \\
			-1 & -1 & 0
		\end{bmatrix}, \quad 
		B = \begin{bmatrix}
			0 & 1 & 0 \\
			0 & 0 & 0 \\
			0 & 0 & 0
		\end{bmatrix}.
		\]
		
		Here, \( A \) is an \( \mathbf{E_{0_2}} \)-matrix and \( B \) is a nonnegative matrix. However, their sum
		\[
		A + B = 
		\begin{bmatrix}
			0 & 0 & -1 \\
			-1 & 0 & -1 \\
			-1 & -1 & 0
		\end{bmatrix}
		\]
		is not an \( \mathbf{E_{0_2}} \)-matrix.
	\end{remark}
	
	Next, we establish that the class of (strictly) semimonotone matrices of exact order \( k \) is closed under transposition and permutation similarity. This invariance ensures that the structural properties characterizing such matrices are preserved under reordering of rows and columns, allowing for flexible analysis and computational approaches.
	
	\begin{proposition}\label{proposition:transpose}
		Let \( A \in \mathbb{R}^{n \times n} \). Then the following statements hold:
		\begin{itemize}
			\item[(a)] \( A \in \mathbf{E_{0_k}} \) (resp., \( \mathbf{E_k} \)) if and only if \( A^T \in \mathbf{E_{0_k}} \) (resp., \( \mathbf{E_k} \)).
			
			\item[(b)] Let \( P \in \mathbb{R}^{n \times n} \) be a permutation matrix. Then \( A \in \mathbf{E_{0_k}} \) (resp., \( \mathbf{E_k} \)) if and only if \( PAP^T \in \mathbf{E_{0_k}} \) (resp., \( \mathbf{E_k} \)).
		\end{itemize}
	\end{proposition}
	
	\begin{proof}
		(a) Suppose \( A \in \mathbf{E_{0_k}} \) (resp., \( \mathbf{E_k} \)). Then every principal submatrix of \( A \) of order \( n-k \) is (strictly) semimonotone, and \( A \) and all its principal submatrices of order \( r > n-k \) are not (strictly) semimonotone. Since the transpose of a principal submatrix of \( A \) is equal to the principal submatrix of \( A^T \), it follows from \cite[Theorem 3.1(6)]{tsatsomeros2019semimonotone} that every principal submatrix of \( A^T \) of order \( n - k \) is (strictly) semimonotone, and none of order \( r > n - k \) is. Hence, \( A^T \in \mathbf{E_{0_k}} \) (resp., \( \mathbf{E_k} \)). The converse follows since \( A = (A^T)^T \).
		
		\vspace{1ex}
		(b) Let \( P \in \mathbb{R}^{n \times n} \) be a permutation matrix and suppose \( A \in \mathbf{E_{0_k}} \) (resp., \( \mathbf{E_k} \)). Then every principal submatrix of \( A \) of order \( n-k \) is (strictly) semimonotone, and none of order \( r > n-k \) is.
		
		Permutation similarity via \( P \) reorders the rows and columns of \( A \) without altering its structural properties. 
		In particular, for any principal submatrix of \( A \) formed by selecting rows and columns indexed by a set \( \alpha \subseteq \{1, \dots, n\} \), there exists a corresponding principal submatrix of \( PAP^T \), formed by selecting rows and columns indexed according to the permuted positions of \( \alpha \), 
		which is permutation similar to \( A_{\alpha\alpha} \). 
		Thus, the semimonotonicity properties of the principal submatrices are preserved under permutation similarity (see \cite[Theorem~4.3]{tsatsomeros2019semimonotone}), and it implies that \( PAP^T \in \mathbf{E_{0_k}} \) (resp., \( \mathbf{E_k} \)). 
		
		Conversely, if \( PAP^T \in \mathbf{E_{0_k}} \) (resp., \( \mathbf{E_k} \)), then since \( A = P^T (PAP^T) P \), the same reasoning applies, and we conclude \( A \in \mathbf{E_{0_k}} \) (resp., \( \mathbf{E_k} \)).
	\end{proof}
	
	The following result shows that the class of (strictly) semimonotone matrices of exact order \( k \) is invariant under scaling by a diagonal matrix with strictly positive diagonal entries.
	
	\begin{proposition}
		Let \( A \in \mathbb{R}^{n \times n} \), and let \( D \) be a diagonal matrix of order \( n \) with strictly positive diagonal entries. Then the following statements are equivalent:
		\begin{itemize}
			\item[(a)] \( A \in \mathbf{E_{0_k}} \) (resp., \( \mathbf{E_k} \)).
			\item[(b)] \( DA \in \mathbf{E_{0_k}} \) (resp., \( \mathbf{E_k} \)).
			\item[(c)] \( AD \in \mathbf{E_{0_k}} \) (resp., \( \mathbf{E_k} \)).
		\end{itemize}
	\end{proposition}
	
	\begin{proof}
		\noindent
		\((a) \Rightarrow  (b).\) Suppose \( A \in \mathbf{E_{0_k}} \) (resp., \( \mathbf{E_k} \)). Then all principal submatrices of \( A \) of order \( n-k \) are (strictly) semimonotone, while \( A \) itself and all its principal submatrices of order \( r > n-k \) are not (strictly) semimonotone. 
		
		Since the principal submatrix of the matrix \( DA \) corresponding to any index set is equal to the product of the corresponding principal submatrices of \( D \) and \( A \), it follows from \cite[Theorem 4.6]{tsatsomeros2019semimonotone} that every principal submatrix of \( DA \) of order \( n - k \) is (strictly) semimonotone, and none of order \( r > n - k \) is. Hence, \( DA \in \mathbf{E_{0_k}} \) (resp., \( \mathbf{E_k} \)).
		
		\vspace{1ex}
		\noindent
		\((b) \Rightarrow  (a).\) Since \( D \) is an invertible diagonal matrix with strictly positive diagonal entries, we can write \( A = D^{-1} (DA) \). Applying the result of part $(a)$, the claim follows.
		
		\vspace{1ex}
		\noindent
		\((b) \Leftrightarrow  (c).\) This equivalence follows directly from Proposition~\ref{proposition:transpose}$(a)$, since \( (AD)^T = D A^{T} \), and the class \( \mathbf{E_{0_k}} \) (resp., \( \mathbf{E_k} \)) is closed under transposition.
	\end{proof}
	
	In the following result, we show that \( n\times n \) (strictly) semimonotone matrices of exact order \( k \) always have nonnegative diagonal entries (resp., positive). 
	Furthermore, we show that, under the additional condition \( n = k + 1 \), such matrices form a subclass of \( \mathbf{Z} \)-matrices.
	
	\begin{theorem}
		Let \( n \geq 3 \) and \( 1 \leq k < n \), and let \( A \in \mathbb{R}^{n \times n} \). Then the following statements hold:
		\begin{itemize}
			\item[(a)] If \( A \in \mathbf{E_{0_k}} \) (resp., \( \mathbf{E_k} \)), then all diagonal entries of \( A \) are nonnegative (resp., positive).
			
			\item[(b)] If \( A \in \mathbf{E_{0_k}} \) (resp., \( \mathbf{E_k} \)) and \( n = k + 1 \), then all off-diagonal entries of \( A \) are negative; that is, \( A \) is a \( \mathbf{Z} \)-matrix.
		\end{itemize}
	\end{theorem}
	
	\begin{proof}
		Suppose \( A \in \mathbb{R}^{n \times n} \) is an \( \mathbf{E_{0_k}} \) (resp., \( \mathbf{E_k} \))-matrix. Then every principal submatrix of \( A \) of order $n-k$ is (strictly) semimonotone. Note that $n-k \geq 1$ as $k < n$. Therefore, all principal submatrices of \( A \) of order $1$ are (strictly) semimonotone, which implies that all diagonal entries of $A$ are (positive) nonnegative. This proves part $(a)$.  
		
		Next, when \( n = 1 + k \) and \( A \in \mathbb{R}^{n \times n} \) is an \( \mathbf{E_{0_k}} \) (resp., \( \mathbf{E_k} \))-matrix. Then every principal submatrix of \( A \) of order $n-k$ is (strictly) semimonotone, but no principal submatrix of order \( r \), where \( 1 = n-k < r \leq n \), is (strictly) semimonotone. It implies that every principal submatrix of \( A \) of order $2$ is almost (strictly) semimonotone. Theorem~\ref{preliminary:theo3.3} shows that each principal submatrix of \( A \) of order $2$ has negative off-diagonal entries. Since each off-diagonal entry of $A$ is an off-diagonal entry of some principal submatrix of $A$ of order $2$. Therefore, all off-diagonal entries of \( A \) are negative. In this case, $A$ is a $\mathbf{Z}$-matrix. This proves part $(b).$
	\end{proof}
	
	\begin{remark}
		As noted in \cite[Theorem 3.1(3)]{tsatsomeros2019semimonotone}, every (strictly) copositive matrix is (strictly) semimonotone. This raises the question: Is this implication preserved when restricted to matrices of general exact order $k$?
		
		The following example shows that the answer is negative.
	\end{remark}
	
	\begin{example}
		Consider the matrices
		\[
		A =
		\begin{bmatrix}
			0 & -1 & -1\\
			0 & 0 & -1 \\
			0 & 0 & 0
		\end{bmatrix} \quad \text{and} \quad
		B =
		\begin{bmatrix}
			1 & -3 & -3\\
			0 & 1 & -3 \\
			0 & 0 & 1
		\end{bmatrix}.
		\]
		The matrix \( A \) is copositive of exact order two but not semimonotone of exact order two, while the matrix \( B \) is strictly copositive of exact order two but not strictly semimonotone of exact order two.
	\end{example}
	
	However, under the symmetry assumption, we now establish that copositivity and semimonotonicity coincide in the case of exact order $k$.
	
	\begin{theorem}
		Let \( A \in \mathbb{R}^{n \times n} \) be a symmetric matrix. Then, the class of (strictly) copositive matrices of exact order $k$ is identical to the class of (strictly) semimonotone matrices of exact order $k$.
	\end{theorem}
	
	\begin{proof}
		Let \( A \in \mathbb{R}^{n \times n} \) be a symmetric (strictly) copositive matrix of exact order \(k\). Then, every principal submatrix of \( A \) of order \( n-k \) is (strictly) copositive. Since every (strictly) copositive matrix is also (strictly) semimonotone \cite[Theorem 3.1(3)]{tsatsomeros2019semimonotone}, it follows that each such submatrix is (strictly) semimonotone. Furthermore, since \( A \) is (strictly) copositive of exact order \(k\), no principal submatrix of order \( r \), where \( n - k < r \leq n \), is (strictly) copositive. Under symmetry, this implies that these submatrices are also not (strictly) semimonotone (see \cite[Propositions 3.9.8 and 3.9.14]{cottle2009linear}). Hence, \( A \) is (strictly) semimonotone of exact order \(k\). The reverse inclusion follows by a similar argument.
	\end{proof}
	
	
	\section*{Declaration of Competing Interest}	
	The authors declare no competing interests.
	
	\section*{Data Availability}
	No data was used in this research.

	\section*{Acknowledgements}
	
	The authors thank Prof.\ K.\ C.\ Sivakumar (Department of Mathematics, Indian Institute of Technology Madras) for carefully reading an earlier draft of this manuscript and for providing valuable comments. The authors also thank the anonymous referees and the Handling Editor for their constructive suggestions, which helped improve the paper's presentation. The first author was supported by the IIT Gandhinagar Post-Doctoral Fellowship (Project No.\ IP/IP/50020).
	

\end{document}